\documentclass{amsart}
\usepackage{pdfsync}
\usepackage{amssymb,amsmath,latexsym,amstext}
\usepackage{graphicx}
\usepackage{epsfig}
\usepackage[english]{babel}

\usepackage{cancel}

\newcommand{\R}{\mathbb R}

\newcommand{\N}{\mathbb N}
\newcommand{\C}{\mathbb C}

\newcommand{\D}{\displaystyle}

\newcommand{\Om}{\Omega}
\newcommand{\Div}{\text{div}}
\newcommand{\pa}{\partial}

\newcommand{\hz}{\hat{z}}
\newcommand{\hp}{\hat{\varphi}}

\newtheorem{thm}{Theorem}[section]

\newtheorem{lem}[thm]{Lemma}
\newtheorem{prop}[thm]{Proposition}
\theoremstyle{definition}

\theoremstyle{remark}
\newtheorem{rem}[thm]{Remark}

\numberwithin{equation}{section}

\begin{document}

\title[The turnpike problem]{TURNPIKE PROPERTY FOR TWO-DIMENSIONAL NAVIER-STOKES EQUATIONS*}

\author[S. Zamorano]{S. Zamorano$^1$}
\footnotetext[1]{Departamento de Ingenier\'ia Matem\'atica,
Universidad de Chile (UMI CNRS 2807), Avenida Beauchef 851, Ed. Norte, Casilla 170-3, Correo 3, Santiago, Chile \\
E-mail: {\it szamorano@dim.uchile.cl}}

\thanks{*Supported by CONICYT Doctoral fellowship 2012-21120662. This work was achieved while I was visiting the BCAM-Basque Center for Applied Mathematics in Bilbao and was partially supported by the Advanced Grant NUMERIWAVES/FP7-246775 of the European Research Council Executive Agency, 
the FA9550-15-1-0027 of AFOSR, the MTM2011-29306 and MTM2014-52347 Grants of the MINECO}

\begin{abstract}

This paper is devoted to the study of the turnpike phenomenon arising in the optimal distributed control tracking-type problem for the Navier-Stokes equations. We obtain a positive answer to this property in the case when the controls are time-dependent functions, and also when are independent of time. In both cases we prove an exponential turnpike property assuming that the stationary optimal state satisfy certain properties of smallness. 
\end{abstract}

\maketitle

\section{Introduction}

In this article, we are dealing with optimal control problem of the incompressible Navier-Stokes equations in two dimensions, both evolutionary and stationary problem. We try to understand what is the relationship between the optimal solution of the nonstationary and the stationary problem, when the times goes to infinity. Specifically, we want to know the conditions under which the nonstationary optimal control and state converges to the stationary optimal control and state, respectively. 

We consider two cases, when the controls are dependent on time and the case where are independent on time. In both cases the optimal control problem consists in minimizing a functional involving both the control and the measure of the difference between the state and a desired stationary state, and terminal constraint. Then, the main idea of this paper is to prove that the optimal controls achieve to get the target state and remains on this situation most of the time. 

In the first case, we establish a result of exponential convergence of the optimality systems associated to the Navier-Stokes equations. We prove, see Theorem \ref{teo10} in Section $5$, under some appropriate smallness conditions of the optimal solutions for the stationary problem, that both optimal evolutionary state and control converge to the respective optimal stationary control and state in a local sense with a exponential rate.

For the second case, as we consider time-independent controls, using the $\Gamma$-convergence we prove that the accumulation point of a sequence of controls for the evolutionary optimal control problem is an optimal control for the stationary problem, see Theorem \ref{teo15} in Section $6$. In this case, we need to ensure the exponential stabilization of the solution of the nonstationary Navier-Stokes problem to the solution of the stationary Navier-Stokes equation, under some smallness condition. 

The smallness condition for the optimal state of the stationary equations is because it is well known that the solution of the stationary Navier-Stokes system is unique when the viscosity is large enough with respect to the right hand side \cite{temam2001navier}. If we remove this condition we need to work with solutions for which the equation is locally unique. These solutions are called \emph{nonsingular solutions}, see, for instance, Casas et al. \cite{casas2007error}.

The study of this type of relationship is commonly used in many models of the fluid mechanics, where the stationary model is considered instead of the evolutionary system. Namely, the underlying idea is that when the time horizon is large enough, the evolutionary optimal control are sufficiently close to the stationary optimal control.

For example, in aeronautics most of the techniques to solve shape optimization are based on stationary models. In that case, is assumed or understood that the optimal shape is close enough to the evolutionary optimal shape, see, for instance, \cite{huan1994optimum}. There are no results justifying such assumptions, specially for models from the fluids mechanics, as the Navier-Stokes or Euler equations (see \cite{jameson2010optimization}).

A recent answer to this problem is given in \cite{porretta2013long}. The authors examined such questions in the context of linear control problems both in the finite dimensional case as infinite dimensional systems, including the linear heat and wave equations. They proved, under suitable observability and controllability assumptions, that optimal controls and state converge exponentially when the time is sufficiently large, to the corresponding stationary case. Porretta and Zuazua in \cite{porretta2013long} mentioned that this type of property in the economy field, specifically in econometry, is known as the \emph{turnpike property}, concept introduced by P. Samuelson. In \cite{trelat2015turnpike}, the authors proved the turnpike property in the case of nonlinear optimal control problem in the finite-dimensional case.

Also, this type of approach can be observed in optimal design. We mention \cite{allaire2010long}, where the autohrs proved that when the time tends to infinity, the optimal design of coefficients of parabolic dynamics converge to those of the elliptic steady state problem. This approach use the classical $\Gamma$-convergence, because they consider coefficients which are independent of time. 

In this work, we consider the Navier-Stokes equations in two dimensions. Navier-Stokes equations are useful because they describe the physics of many things of scientific and engineering interest. They may be used to model the weather, ocean currents, water flow in a pipe and air flow around a wing. The Navier-Stokes system in their full and simplified forms help with the design of aircraft and cars, the study of blood flow, the design of power stations, the analysis of pollution, among others. 

It is well known that in the three-dimensional cases there are many open problems connected with smoothness and uniqueness of weak solutions, both nonstationary as stationary models. Hence, in this paper we restrict our attention to the two-dimensional case. In particular, we consider incompressible and newtonian fluids. Namely, the density remains constant within a parcel of fluid that moves with the flow velocity and constant viscosity, respectively. Obviously, the next step is consider both Euler as Navier-Stokes equations for compressible and viscous fluids. This type of fluids are more realistic in the field of aeronautic (see \cite{jameson1998optimum}), but this models could be much more complex.

The literature of optimal control problem for the Navier-Stokes equations are very extensive. We mention the work of \cite{abergel1990some} for evolution optimal control problems in fluids mechanics in the case of two-dimensional flows. Also, the PhD thesis \cite{hinze2000optimal} and \cite{wachsmuth2006optimal} for Navier-Stokes equations. In the stationary case, we refer \cite{abergel1993some} and \cite{de2004primal}, and the references therein. 

Since the main result of this paper is in a local sense, for technical reasons, we need some properties about the linearized Navier-Stokes equations. This equation is known as \emph{Oseen equation} or \emph{Stokes-Oseen equation}. The importance of this equation for the study of the Navier-Stokes system is fundamental, specially for the feedback stabilization of the Navier-Stokes problem around an unstable stationary solution, see \cite{barbu2003feedback, fursikov2001stabilizability, fursikov2012feedback, raymond2006feedback}. In our case, the Oseen equation is fundamental to obtain a positive response on the turnpike property for the Navier-Stokes problem.

The outline of the paper is a follows. In Section $2$ we introduce some function space according to the theory for Navier-Stokes equations, and present the basic result of existence and uniqueness for the state equations, both the evolutionary problem and stationary. Also, we gives the basic properties of the Oseen equation. In Section $3$ we formulate the optimal control problem for both nonstationary problem and stationary, and present existence results, first order necessary and second order conditions. In Section $4$, we state and prove the main result of the paper, see Theorem \ref{teo10}. Finally, in Section $5$, we prove a turnpike property in the special case when the controls are independent of time.

\setcounter{equation}{0}
\section{Mathematical Setting}

In this section, we provide some functional analytic background to study the Navier-Stokes equations. Here, we rely on the book by Temam \cite{temam2001navier}. The existence, uniqueness and regularity of weak solutions to the Navier-Stokes equations is completely understood in the two-dimensional case. Hence, we focus our work to the two-dimensional case.

Let $\Om\subset \R^2$ be a bounded and simply connected domain, with boundary $\pa\Om$ of class $C^2$. 
Following Temam \cite{temam2001navier}, we set
\begin{align*}
V=\{v\in (H_{0}^{1}(\Om))^2:\; \Div \; v=0\}, \;\; H=\{v\in (L^2(\Om))^2:\;\Div\; v=0,\; \gamma_{n}u=0\},
\end{align*}
where $\gamma_{n}$ denotes the normal component of the trace operator. 

The spaces $V,H$, and $V'$ satisfies
\begin{align*}
V\subset H=H'\subset V'
\end{align*}
with dense and continuous imbedding.  

Let us introduce a trilinear form $b: V\times V\times V\to \R$ as the variational formulation of the nonlinearity term $(u\cdot\nabla)v$ by
\begin{align*}
b(u,v,w)=\D\int_{\Om}((u\cdot\nabla v))\cdot w dx.
\end{align*}

We know  that the trilinear form $b$ satisfies the properties of Lemma \ref{lema3.1}. This properties are fundamental for the study of the Navier-Stokes equations, and will be used throughout the paper.
\begin{lem}[see Chapter III, \cite{temam2001navier}]\label{lema3.1}

\smallskip\

\begin{enumerate}
\item $b(u,v,w)+b(u,w,v)=0$, $\forall u\in V$, $\forall v,w\in (H^{1}(\Om)^2$.
\item $b(u,v,v)=0$, $\forall u\in V$, $\forall v\in (H^{1}(\Om)^2$.
\item $b(u,v,w)=((\nabla v)^{T}w,u)$, $\forall u,v,w \in (H^{1}(\Om)^2$. 
\item For all $u\in V$ and all $v,w\in (H^{1}(\Om)^2$ we have

\begin{align}\label{2.1}
|b(u,v,w)|\leq C \|u\|_{L^2(\Om)}^{1/2}\|u\|_{H^1(\Om)}^{1/2}\|v\|_{L^2(\Om)}^{1/2}\|v\|_{H^1(\Om)}^{1/2}\|w\|_{H^1(\Om)}.
\end{align}
\end{enumerate}
\end{lem}

Let $A$ be a operator defined as follows: $Ay=-P(\Delta y)$, where $\Delta$ is the vector Laplacian, and $P$ is the orthogonal projector from $(L^2(\Om))^2$ onto $H$, called the Leray projector. And let $B$ be the nonlinear operator from $V$ into its dual $V'$, such that 
\begin{align*}
\D\langle By,v\rangle_{V',V}=b(y,y,v)=\D\int_{\Om}((y\cdot \nabla) y)v dx, \quad\forall v\in V.
\end{align*}

Concerning the operator $B$, we have the following properties for the differentiability that we use throughout this work.
\begin{prop}[see \cite{abergel1990some}]

\smallskip\

\begin{enumerate}
\item $y\to B(y)$ is differentiable from $V$ into $V'$, and we have
\begin{align*}
\langle B'(y)v,w\rangle=b(y,v,w)+b(v,y,w).
\end{align*}
\item Let $B'(y)^{*}$ denote the adjoint of $B'(y)$ for the duality between $V$ and $V'$, then we have
\begin{align*}
\langle B'(y)^{*}v,w\rangle=b(w,y,v)-b(y,v,w).
\end{align*}
\end{enumerate}
\end{prop}

\subsection{Nonstationary Navier-Stokes Problem}

Given $T>0$, we denote $\Om_{T}=\Om\times(0,T)$ and $\Gamma_{T}=\pa\Om\times(0,T)$. Under the previous framework, we consider the incompressible Navier-Stokes problem

\begin{align}\label{2.2}
\left\{\begin{array}{rllll}
y_{t}-\mu\Delta y+(y\cdot\nabla)y+\nabla p&=&u& ,& \text{in }Q_{T},\\
\Div \ y&=&0&,&\text{in } Q_{T},\\
y&=&0&,&\text{on }\Gamma_{T},\\
y(x,0)&=&y_0(x)&,& x\in\Om,
\end{array}
\right.
\end{align}
where the forcing term $u$ is in $L^2(0,T; H)$, the initial data $y_0$ is in $H$, and the kinematic viscosity $\mu>0$.

The Navier-Stokes equations \eqref{2.2} in $\Om$ can be written under the following form, see \cite{temam2001navier},
\begin{align}\label{2.3}
\left\{\begin{array}{rllll}
\D\frac{dy(t)}{dt}+\mu Ay(t)+By(t)&=&u(t)&,&t\leq 0,\\
y(0)&=&0&.&
\end{array}
\right.
\end{align}

This variational formulation, excluding the pressure, of the Navier-Stokes problem is by now classical. We recall the classical result of existence and uniqueness of weak solutions related to \eqref{2.2}.

\begin{thm}[see Chapter III, \cite{temam2001navier}]\label{teo1}
For any given $u\in L^2(0,T; (L^2(\Om))^2)$ and $y_0\in V$, there exists a unique weak solution of \eqref{2.2} satisfying for all $T>0$
\begin{align*}
(y,p)\in(C([0,T];V)\cap L^2(0,T;(H^2(\Om))^2\cap V))\times L^{2}(0,T; H^1(\Om)\cap L_0^2(\Om)),
\end{align*}
and
\begin{align*}
y_{t}\in L^2(0,T; H).
\end{align*}
Moreover, it satisfies the following energy equality for all $t\in \R^{+}$,
\begin{align*}
\D\frac{1}{2}\|y(t)\|_{L^2(\Om)}^{2}+\mu\int_{0}^{t}\|\nabla y(\tau)\|_{L^2(\Om)}^2 d\tau=\frac{1}{2}\|y_0\|_{L^2(\Om)}^2+\int_{0}^{t}\langle u(\tau),y(\tau)\rangle_{L^2(\Om)}d\tau.
\end{align*}
\end{thm}

\subsection{Stationary Navier-Stokes Problem}

Now we give the basic result for the existence and uniqueness for the stationary Navier-Stokes problem. We consider the following problem

\begin{align}\label{2.11}
\left\{\begin{array}{rllll}
-\mu \Delta y+(y\cdot\nabla)y+\nabla p&=&u& ,& \text{in }\Om,\\
\Div \ y&=&0&,&\text{in } \Om,\\
y&=&0&,&\text{on }\pa\Om,
\end{array}
\right.
\end{align}
where $u\in (L^2(\Om))^2$.

Under certain conditions of smallness, we obtain the following result of existence and uniqueness of weak solutions of \eqref{2.11}.

\begin{thm}[see Chapter II, \cite{temam2001navier}]\label{teo5}
If $\|u\|_{V'}\leq C(\Om) \mu^2$, then the problem \eqref{2.11} has a unique weak solution
\begin{align*}
y\in H^2(\Om)\cap V \ , \quad p\in H^1(\Om).
\end{align*}

Moreover, $y$ satisfies the following estimate
\begin{align}\label{2.12}
\|y\|_{V}\leq \frac{1}{\mu}\|u\|_{V'}.
\end{align}
\end{thm}

\begin{rem}
The constant $C=C(\Om)$ in Theorem $\ref{teo5}$ is given by
$$
C=\D\sup_{l,v,w\in V}\frac{|b(l,v,w)|}{\|l\|_{V}\|v\|_{V}\|w\|_{V}}
$$
\end{rem}

\begin{rem}
The smallness condition on the right hand side guarantees the uniqueness of solutions of \eqref{2.11}, see, for instance, Temam \cite{temam2001navier}. 
\end{rem}

\subsection{Oseen equation}

We will need in the following some results about the linearized equations. In the literature, this problem is so-called \emph{Oseen equation}, and in the Barbu book \cite{barbu2011stabilization} is called \emph{Stokes-Oseen equation}. We refer the reader to the extensive survey \cite{barbu2011stabilization, galdi2013introduction} and references therein.

Give a state $\overline{y}\in (H^2(\Om))^2\cap V$, solution of the steady state Navier-Stokes problem \eqref{2.11}, we consider the linearized Navier-Stokes equation around the state $\overline{y}$

\begin{align}\label{4.1}
\left\{\begin{array}{rllll}
w_{t}-\mu\Delta w+(\overline{y}\cdot\nabla)w+(w\cdot\nabla)\overline{y}+\nabla p&=&f& ,& \text{in }Q_{T},\\
\Div \ w&=&0&,&\text{in } Q_{T},\\
w&=&0&,&\text{on }\Gamma_{T},\\
w(x,0)&=&w_0(x)&,& x\in\Om.
\end{array}
\right.
\end{align}

If $w_0$ is in $V$ and $f$ is in $L^2(0,T; (L^2(\Om))^2)$, then there exists a unique weak solution $(w,p)$ in $(C([0,T];V)\cap L^2(0,T;(H^2(\Om))^2\cap V))\times L^{2}(0,T; H^1(\Om)\cap L_0^2(\Om))$ of \eqref{4.1}. 

For technical reason in the proof of the main result, we need to give some properties for this equation. We define the Oseen operator $\mathcal{A}$ as
\begin{align}\label{4.2}
\mathcal{A}v:=-\mu P(\Delta v)+P[(\overline{y}\cdot\nabla)v+(v\cdot\nabla)\overline{y}],
\end{align}
where $P$ is the Leray projector.

This operator is closed and has the domain $D(\mathcal{A})=D(A)=(H^2(\Om))^2\cap V$, where $A$ is the Stokes operator defined at the beginning. 

Assuming that the spaces are complex, we denote by $\rho(\mathcal{A})$ the resolvent set of operator $\mathcal{A}$, namely, the set of $\lambda\in\C$ such that the resolvent operator
\begin{align*}
R(\lambda,\mathcal{A})\equiv(\lambda I-\mathcal{A})^{-1}
\end{align*}
is defined and continuous. Here $I$ is the identity operator. The complement of $\rho(\mathcal{A})$ is called the spectrum of the operator $\mathcal{A}$ and is denoted by $\Sigma(\mathcal{A})$.

It is well known that for $\lambda\in \rho(\mathcal{A})$, the resolvent of Oseen operator \eqref{4.2} is a compact operator, and the spectrum $\Sigma(\mathcal{A})$ consists of a discrete set of points. Moreover, Oseen operator is sectorial. 

Now, let us consider the adjoint operator $\mathcal{A}^{*}$ to Oseen operator
\begin{align}\label{4.3}
\mathcal{A}^{*}v:=-\mu P(\Delta v)-P[(\overline{y}\cdot\nabla)v-(\nabla\overline{y})^{T}v],
\end{align}
where $T$ denote the transpose of $\nabla\overline{y}$.

Evidently, $\mathcal{A}^{*}$ are the same properties than $\mathcal{A}$. Namely, is closed with domain $D(\mathcal{A}^{*})=(H^2(\Om))^2\cap V$. Moreover, $\mathcal{A}^{*}$ is sectorial with a compact resolvent. Besides, we assume that $\overline{y}\in (H^2(\Om))^2\cap V$, then $\rho(\mathcal{A})=\rho(\mathcal{A}^{*})$.

Let $\sigma>0$ be a constant satisfying
\begin{align}\label{4.4}
\Sigma(\mathcal{A})\cap\{\lambda\in\C\; :\; \text{Re}\lambda=\sigma\}=\emptyset.
\end{align}

Denote by $X_{\sigma}^{+}(\mathcal{A})$ the subspace of $H$ generated by all eigenfunctions and associated functions of operator $\mathcal{A}$ corresponding to all eigenvalues of $\mathcal{A}$ placed in the set $\{\lambda\in\C\;:\; \text{Re}\lambda <\sigma\}$. By $X_{\sigma}^{+}(\mathcal{A}^{*})$ we denote analogous subspace corresponding to adjoint operator $\mathcal{A}^{*}$. We denote the orthogonal complement to $X_{\sigma}^{*}(\mathcal{A}^{*})$ in $H$ by $X_{\sigma}$. Then, we have the following result of Fursikov \cite{fursikov2012feedback}.

\begin{thm}[see \cite{fursikov2012feedback}]\label{teo9}
Suppose that $\mathcal{A}$ is the operator \eqref{4.2} and $\sigma>0$ satisfies \eqref{4.4}. Then for each $w_0\in X_{\sigma}$ we have
\begin{align}\label{4.5}
\D\|w(t,\cdot)\|_{V}\leq c\|w_0\|_{V}\; e^{-\sigma t},\quad \text{for }t\geq 0.
\end{align}
\end{thm}

\setcounter{equation}{0}
\section{Optimal control problem and existence of solutions}\label{section3}

In this section we introduce the optimal control problem for the evolutionary and stationary Navier-Stokes problem in two dimensions. We show the existence of optimal solution and state the theorems about the first-order optimality conditions. Besides, we prove that, in the case when the tracking term is sufficiently small, the second derivative of the functional to minimize is positive definite.

\subsection{Evolutionary optimal control for Navier-Stokes equations}

We recall that our analysis is in two dimension. In this case, see Lemmas \ref{lem1} and \ref{lem2}, the relation between the control and the state is differentiable, which simplifies the analysis for the optimality conditions. For the three-dimensional case is more complicated to derive some optimality conditions. A possibility, as in \cite{casas1998optimal}, is to work with the so-called strong solutions of the Navier-Stokes problem. This type of solution is well known, see, for instance, \cite{boyer2012mathematical} Chapter V.2. The advantage of these solutions is that the uniqueness is known, but the existence is still an open problem.

Let us introduce the optimal control tracking-type problem of the evolutionary Navier-Stokes equations: 

\emph{find $u^{T}\in L^2(0,T;H)$, $y^{T}$ is the solution of \eqref{2.3} associated to $u^{T}$, minimizing the functional}
\begin{align}\label{2.4}
\D J^{T}(u)=\frac{1}{2}\int_{0}^{T}\|y(t)-x^{d}\|_{L^2(\Om)}^2 dt+\frac{k}{2}\int_{0}^{T}\|u(t)\|_{L^2(\Om)}^2 dt+q_0\cdot y(T),
\end{align}
\emph{where $x^{d}\in(L^2(\Om))^2$ is desired state, $q_0\in(L^2(\Om))^2$ and $k>0$ is a constant}.

Let us remark that the controls $u$ can act on all domain $\Om$ or on a subset of $\Om$. 

We observe that the problem \eqref{2.4} is a nonconvex optimization problem because the mapping $u\mapsto y_{u}$ is nonlinear. But we show that if the tracking term, $\|y(t)-x^{d}\|_{L^2(\Om)}$, is sufficiently small, then the Hessian of $J^{T}$ is positive definite.

\begin{thm}\label{teo2}
Let $y_0\in V$. There exists at least an element $u^{T}\in L^2(0,T;H)$, and $y^{T}\in C([0,T]; V)\cap L^2(0,T; (H^2(\Om))^2)$ such that the functional $J^{T}(u)$ attains its minimum at $u^{T}$, and $y^{T}$ is the solution of \eqref{2.2} associated to $u^{T}$.
\end{thm}

\begin{proof}
The functional $J^{T}$ is bounded from below. Hence, there exists the infimum of $J^{T}$. Moreover, let us take a minimizing sequence $(y_{n},u_{n})$. Since
\begin{align*}
\D\frac{k}{2}\int_{0}^{T}\|u_{n}\|_{L^2(\Om)}dt\leq J^{T}(u_{n})<\infty,
\end{align*}
we deduce that $(u_{n})$ is bounded in $L^2(0,T;H)$ and, consequently, $(y_{n})$ is bounded in $C([0,T]; V)\cap L^2(0,T; (H^2(\Om))^2)$ as well. Therefore, we can extract a subsequence, denoted in the same way, converging weakly in \\$L^2(0,T; ( H^2(\Om))^2)\times L^2(0,T;H)$ to $(y^{*},u^{*})$. 

Now, we need to prove that the pair $(y^{*},u^{*})$ satisfies the equation \eqref{2.2}. The only problem is to pass to the limit in the nonlinear term $(y_{n}\cdot\nabla)y_{n}$. By the result in Chapter III in \cite{temam2001navier}, we obtain a compactness property, this implies that $y_{n}\to y^{*}$ strongly in $L^2(0,T; H)$. By Lemma $3.2$ Chapter III in \cite{temam2001navier}, we obtain that
\begin{align*}
b(y_{n},y_{n},v) \to b(y^{*}y^{*},v), \text{ as }n\to\infty.
\end{align*}

Then, taking into account the linearity and continuity of the other terms involved, the limit $(y^{*},u^{*})$ satisfies the state equations.

Finally, the objective functional consists of several norms, thus it is weakly semicontinuous which implies
\begin{align*}
J^{T}(u^{*})\leq \liminf J^{T}(u_{n})=\inf J^{T}(u).
\end{align*}

Therefore, $u^{*}$ is an optimal solution, with $y^{*}$ the solution of \eqref{2.2} associated to $u^{*}$.
\end{proof}

\subsubsection{First-Order necessary optimality conditions}

We now proceed to derive the first-order optimality conditions associated with the problem \eqref{2.4}. This is done by studying the G\^ateaux derivative of the functional $J^{T}(u)$.

We will need, in the following, some results about the so-called control-to-state mapping. The next two lemmas can be found in \cite{abergel1990some}.

\begin{lem}\label{lem1}
Let $y_0$ be in $V$. The mapping $u\mapsto y_{u}$, from $L^2(0,T;H)$ into $L^2(0,T;V)$, has a G\^ateaux derivative $((\frac{Dy_{u}}{Du})\cdot h)$ in every direction $h_1$ in $L^2(0,T;H)$. Furthermore, $(\frac{Dy_{u}}{Du})\cdot h_1=w(h_1)$ is the solution of the linearized problem
\begin{align}\label{2.5}
\left\{\begin{array}{rllll}
\D\frac{dw}{dt}+\mu Aw+B'(y_{u})\cdot w&=&h_1&,&t\leq 0,\\
w(0)&=&0&.&
\end{array}
\right.
\end{align}

Finally, $w$ is in $L^{\infty}(0,T;V)\cap L^2(0,T;(H^2(\Om))^2)$ and $\|B'(y_{u})w\|_{L^2(V')}\leq c\|y_{u}\|\|w\|$.
\end{lem}

\begin{lem}\label{lem2}
Let $h_1$ be given in $L^2(0,T;H)$, and let $w(h_1)$ be defined as above. Then, for every $h_2$ in $L^2(0,T;H)$ we have
\begin{align*}
\D\iint_{Q_{T}}(h_2\cdot w(h_1))(x,t)dxdt=\iint_{Q_{T}}(\tilde{w}(h_2)\cdot h_1)(x,t)dxdt,
\end{align*}
where $\tilde{w}(h_2)$ is the solution of the adjoint linearized problem
\begin{align}\label{2.6}
\left\{\begin{array}{rllll}
\D-\frac{d\tilde{w}}{dt}+\mu A\tilde{w}+B'(y_{u})^{*}\cdot \tilde{w}&=&h_2&,&t\leq 0,\\
\tilde{w}(T)&=&0&.&
\end{array}
\right.
\end{align}
\end{lem}

\begin{rem}
Writing systems \eqref{2.5} and \eqref{2.6} in a extended way, it is possible to express $w$ and $\tilde{w}$ as the respective solutions of the following equations:
\begin{align}\label{2.5'}
\left\{\begin{array}{rllll}
w_{t}-\mu\Delta w+(y_{u}\cdot\nabla)w+(w\cdot\nabla)y_{u}+\nabla p&=&h_1& ,& \text{in }Q_{T},\\
\Div \ w&=&0&,&\text{in } Q_{T},\\
w&=&0&,&\text{on }\Gamma_{T},\\
w(x,0)&=&0&,& x\in\Om,
\end{array}
\right.
\end{align}
and
\begin{align}\label{2.6'}
\left\{\begin{array}{rllll}
-\tilde{w}_{t}-\mu\Delta \tilde{w}+(\nabla y_{u})^{T}\tilde{w}-(y_{u}\cdot\nabla)\tilde{w}+\nabla \tilde{p}&=&h_2& ,& \text{in }Q_{T},\\
\Div \ \tilde{w}&=&0&,&\text{in } Q_{T},\\
\tilde{w}&=&0&,&\text{on }\Gamma_{T},\\
\tilde{w}(x,T)&=&0&,& x\in\Om.
\end{array}
\right.
\end{align}
\end{rem}

Using the last two Lemmas, Abergel and Temam \cite{abergel1990some} prove the following first-order optimality condition for the optimal control problem \eqref{2.4}. The proof can be obtained by the usual approach. 

\begin{thm}[see \cite{abergel1990some}]\label{teo3}
Let $(y^{T},u^{T})$ be an optimal pair for problem \eqref{2.4}. The following equality holds
\begin{align*}
u^{T}+q=0,
\end{align*}
where $q$ is the adjoint state that is the solution of the linearized adjoint problem
\begin{align}\label{2.7}
\left\{\begin{array}{rllll}
-q_{t}-\mu\Delta q+(\nabla y^{T})^{T}q-(y^{T}\cdot\nabla)q+\nabla \tilde{p}&=&y^{T}-x^{d}& ,& \text{in }Q_{T},\\
\Div \ q&=&0&,&\text{in } Q_{T},\\
q&=&0&,&\text{on }\Gamma_{T},\\
q(x,T)&=&q_0&,& x\in\Om.
\end{array}
\right.
\end{align}

Moreover, $u^{T}$ is in $L^{\infty}(0,T;V)\cap L^2(0,T;(H^2(\Om))^2)$.
\end{thm}

\subsubsection{Second order conditions}

In the following result we assert positive definiteness of the Hessian provided that $\|y-x^{d}\|_{L^2(0,T;V)}$ is sufficiently small, a condition which is applicable to tracking type problems. 

We observe that in \cite{porretta2014remarks} the authors proved that the functional to minimize is also positive definite at least when the target and the initial data are small enough.

\begin{thm}\label{teo4}
If $\|y-x^{d}\|_{L^2(0,T;V)}$ is sufficiently small, then the Hessian $J^{T}(u)''$ is positive definite. 
\end{thm}

\begin{proof}
By Chapter $2$ of \cite{hinze2000optimal}, we have that the second G\^ateaux derivative of $J^{T}$ is given by
\begin{align}
J^{T}(u)''v^2=\iint_{Q_{T}}|y_{v}|^2dxdt+\iint_{Q_{T}}|v|^2dxdt-2\iint_{Q_{T}}(y_{v}\cdot\nabla)y_{v}\cdot q_{u} dxdt,
\end{align}
where $y_{v}$ is the solution of the linearized equation
\begin{align}\label{2.9}
\left\{\begin{array}{rllll}
\D\frac{dy_{v}}{dt}+\mu Ay_{v}+B'(y_{u})\cdot y_{v}&=&v&,&t\leq 0,\\
v(0)&=&0&,&
\end{array}
\right.
\end{align}
in the direction $v$, and $q_{u}$ the solution of the adjoint linearized problem
\begin{align}\label{2.10}
\left\{\begin{array}{rllll}
\D-\frac{dq_{u}}{dt}+\mu Aq_{u}+B'(y_{u})^{*}\cdot q_{u}&=&y_{u}-x^{d}&,&t\leq 0,\\
q_{u}(T)&=&q_0&.&
\end{array}
\right.
\end{align}

Since $B$ is of quadratic nature, we have that the second derivative of $B$ is given by
\begin{align*}
\D B''(y)y_{v}^2=B'(y_{v})y_{v}=2B(y_{v}).
\end{align*}

Besides, we known that
\begin{align*}
\|B'(y_{v})y_{v}\|_{L^2(0,T;V')}\leq C\|y_{v}\|_{L^2(0,T;V)}^2.
\end{align*}
Moreover, the solution of the linearized equation \eqref{2.9} satisfy
\begin{align*}
\|y_{v}\|_{L^2(0,T;V)}\leq C\|v\|_{L^2(0,T;V)}.
\end{align*}

For the adjoint linearized problem \eqref{2.10} we obtain that
\begin{align*}
\|q_{u}\|_{L^2(0,T;V)}\leq C\|y_{u}-x^{d}\|_{L^2(0,T;V)}.
\end{align*}

Then, we conclude that the second derivative of $J^{T}$ can be estimated as
\begin{align*}
\D J^{T}(u)''v^2\geq \iint_{Q_{T}}|y_{v}|^2dxdt+(1-C\|y_{u}-x^{d}\|_{L^2(0,T;V)})\iint_{Q_{T}}v^2 dxdt,
\end{align*}
which gives the assertion.
\end{proof}

\smallskip\

\subsection{Stationary optimal control problem for Navier-Stokes equations}

As for the nonstationary Navier-Stokes equations, our optimal control problem is to find $\overline{u}$, $\overline{y}$ being the solution of \eqref{2.11} associated to $\overline{u}$, minimizing the functional
\begin{align}\label{2.13}
\D J(u)=\frac{1}{2}\|y-x^{d}\|_{L^2(\Om)}^2 +\frac{\alpha}{2}\|u\|_{L^2(\Om)}^2,
\end{align}
where $x^{d}\in(L^2(\Om))^2$ is a target and $\alpha>0$ is a constant.

We are going to show that the optimal control problem \eqref{2.13} has a solution.

\begin{thm}\label{teo6}
There exists at least an element $\overline{u}\in L^2(\Om)$, and $\overline{y}\in H^2(\Om)\cap V$ solution of \eqref{2.11} associated to $\overline{u}$, such that the functional $J(u)$ attains its minimum at $\overline{u}$.
\end{thm}

\begin{proof}
The functional $J$ is bounded below by zero. Then we can take a minimizing sequence $(y_{n},u_{n})$. Is easy to see that $\frac{\alpha}{2}\|u_{n}\|^2\leq J(u_{n})<\infty$, which implies that the sequence $(u_{n})$ is uniformly bounded in $L^2(\Om)$. 

From the regularity of the Navier-Stokes problem we obtain that the sequence $(y_{n})$ is uniformly bounded in $H^2(\Om)\cap V$, and then implies that we can extract a weakly  convergent subsequence, denoted in the same way $(y_{n},u_{n})$, such that 
\begin{align*}
y_{n}\rightharpoonup y^{*} \text{ in }H^2(\Om)\cap V ,\quad u_{n}\rightharpoonup u^{*} \text{ in }L^2(\Om).
\end{align*}

Now, we need to ensure that $(y^{*},u^{*})$ is a solution of the Navier-Stokes problem. For this steep we use the trilinear continuous form $b$. Thanks to the compact embedding $H^2(\Om)\cap V \hookrightarrow V$ and the continuity of $b$, we obtain that $b(y_{n},y_{n},v)\to b(y^{*},y^{*},v)$, as $n\to\infty$. Then, we have that $(y^{*},u^{*})$ satisfies the Navier-Stokes problem. 

Therefore, as $J$ is weakly lower semicontinuous, the result is proved.

\end{proof}

\subsubsection{First-Order necessary optimality conditions}

The following result of J. De los Reyes \cite{de2004primal}, shows the first-order optimality conditions in the case of the stationary Navier-Stokes equations. This theorem is more general, since De los Reyes consider the constrained optimal control problem. He proved the result based on a result of Lagrange multipliers.

\begin{thm}[see \cite{de2004primal}]\label{teo7}
Let $(\overline{u},\overline{y})$ be an optimal solution for \eqref{2.13}, such that $\mu>\mathcal{M}(\overline{y})$, where $\mathcal{M}(y)=\D\sup_{v\in V}\frac{|b(v,v,y)|}{\|v\|_{V}^2}$. Then there exists $q\in V$ such that satisfies the following optimality system in variational sense
\begin{align}\label{2.14}
\left\{\begin{array}{rllll}
-\mu \Delta \overline{y}+(\overline{y}\cdot\nabla)\overline{y}+\nabla \overline{p}&=&-q& ,& \textrm{in }\Om,\\
\Div \ \overline{y}&=&0&,&\text{in } \Om,\\
\overline{y}&=&0&,&\text{on }\pa\Om,\\
\D -\mu \Delta q-(\overline{y}\cdot\nabla)q+(\nabla\overline{y})^{T}q+\nabla\pi&=&\overline{y}-x^{d}&,&\text{in }\Om\\
\Div \ q&=&0&,&\text{in } \Om,\\
q&=&0&,&\text{on }\pa\Om.
\end{array}
\right.
\end{align}

Moreover, $(q,\pi)\in (H^2(\Om))^2\times H^1(\Om)$ and satisfies the estimate
\begin{align}\label{2.15}
\|q\|_{V}\leq \frac{c}{\mu-\mathcal{M}(\overline{y})}\|\overline{y}-x^{d}\|_{L^2(\Om)}.
\end{align}
\end{thm}

\begin{rem}
The assumption $\mu>\mathcal{M}(\overline{y})$ is a sufficient requirement for the satisfaction of the regular point condition, see \cite{zowe1979regularity}. 
\end{rem}

\subsubsection{Second order conditions}

The next result is relevant for our purposes. In the next section we use this result to prove the turnpike property for a particular system, the Oseen equation.

\begin{thm}\label{teo8}
Assume that $\|y-x^{d}\|_{V}$ is sufficiently small and $\mu>\mathcal{M}(\overline{y})$. Then, the Hessian $J(u)''$ is positive definite. 
\end{thm}                                                                                       

\begin{proof}
By Theorem $3.3$ in \cite{casas2007error}, we have that the second G\^ateaux derivative of $J$ is given by
\begin{align}\label{2.16}
J''(u)v^2=\D\int_{\Om}|y_{v}|^2dx+\int_{\Om}|v|^2dx-2\int_{\Om}(y_{v}\cdot\nabla)y_{v}\cdot q_{v}dx,
\end{align}
where $y_{v}$ is the solution of the linearized problem
\begin{align}\label{2.17}
\left\{\begin{array}{rllll}
-\mu\Delta y_{v}+(y\cdot\nabla)y_{v}+(y_{v}\cdot\nabla)y+\nabla p_{v}&=&v& ,& \text{in }\Om,\\
\Div \ y_{v}&=&0&,&\text{in } \Om,\\
y_{v}&=&0&,&\text{on }\pa\Om,\\
\end{array}
\right.
\end{align}
in the direction $v$, and $q_{u}$ the solution of the adjoint linearized problem
\begin{align}\label{2.18}
\left\{\begin{array}{rllll}
-\mu\Delta q_{u}+(\nabla y)^{T}q_{u}-(y\cdot\nabla)q_{u}+\nabla \tilde{p}&=&y-x^{d}& ,& \text{in }\Om,\\
\Div \ q_{u}&=&0&,&\text{in }\Om,\\
q_{u}&=&0&,&\text{on }\pa\Om,\\
\end{array}
\right.
\end{align}
Reasoning as in Theorem \ref{teo4}, using the Theorem \ref{teo8}, we deduce that
\begin{align*}
J''(u)v^2\geq\D\int_{\Om}|y_{v}|^2dx+(1-C\|y-x^{d}\|_{V})\int_{\Om}|v|^2dx,
\end{align*}
which implies the claim.
\end{proof}

\setcounter{equation}{0}
\section{Turnpike property for the two-dimensional Navier-Stokes problem with time-dependent control}

In this section we prove a turnpike result for the optimality system of Navier-Stokes problem, under the condition that the initial and final states are close enough to the stationary primal and dual state, respectively. Also, we need some assumption of smallness for the solution of the stationary adjoint equation.

As in the paper of Porretta and Zuazua \cite{porretta2014remarks}, the smallness condition is to ensure the exponential turnpike property of the linearized optimality system. In \cite{porretta2014remarks}, the authors prove under the smallness of the target and the initial condition that the linearized optimality system satisfies the turnpike property. However, by the quadratic nature of the nonlinear term $B$,  in this paper we only assume the smallness of the tracking term.

From the results of Section \ref{section3}, we have the following optimality system for the nonstationary Navier-Stokes equations (see Theorem \ref{teo3})
\begin{align}\label{3.1}
\left\{\begin{array}{rllll}
y_{t}^{T}-\mu\Delta y^{T}+(y^{T}\cdot\nabla)y^{T}+\nabla p^{T}&=&-q^{T}& ,& \text{in }Q_{T},\\
\Div \ y^{T}&=&0&,&\text{in } Q_{T},\\
y^{T}&=&0&,&\text{on }\Gamma_{T},\\
y^{T}(x,0)&=&y_0(x)&,& x\in\Om,\\
-q_{t}^{T}-\mu\Delta q^{T}-(y^{T}\cdot\nabla)q^{T}+(\nabla y^{T})^{T}q^{T}+\nabla \pi^{T}&=&y^{T}-x^{d}& ,& \text{in }Q_{T},\\
\Div \ q^{T}&=&0&,&\text{in } Q_{T},\\
q^{T}&=&0&,&\text{on }\Gamma_{T},\\
q^{T}(x,T)&=&q_0&,& x\in\Om.
\end{array}
\right.
\end{align}

And, for the stationary Navier-Stokes problem, see Theorem \ref{teo7}, we obtain 
\begin{align}\label{3.2}
\left\{\begin{array}{rllll}
-\mu \Delta \overline{y}+(\overline{y}\cdot\nabla)\overline{y}+\nabla \overline{p}&=&-\overline{q}& ,& \text{in }\Om,\\
\Div \ \overline{y}&=&0&,&\text{in } \Om,\\
\overline{y}&=&0&,&\text{on }\pa\Om,\\
\D -\mu \Delta \overline{q}-(\overline{y}\cdot\nabla)\overline{q}+(\nabla\overline{y})^{T}\overline{q}+\nabla\overline{\pi}&=&\overline{y}-x^{d}&,&\text{in }\Om\\
\Div \ \overline{q}&=&0&,&\text{in } \Om,\\
\overline{q}&=&0&,&\text{on }\pa\Om.
\end{array}
\right.
\end{align}

Now, we develop a local analysis around a given steady state optimal control $(\overline{y},\overline{u})$.

We consider $y=\overline{y}+z$, $p=\overline{p}+\eta$, $q=\overline{q}+\varphi$, and $\pi=\overline{\pi}+\nu$. Then, the optimality system linearized around the stationary solutions takes the form
\begin{align}\label{3.3}
\left\{\begin{aligned}
&z_{t}-\mu\Delta z+(\overline{y}\cdot\nabla)z+(z\cdot\nabla)\overline{y}+\nabla\eta=-\varphi,&&\text{in }Q_{T},\\
&\Div \ z=0,&&\text{in }Q_{T},\\
&z=0,&&\text{on }\Gamma_{T},\\
&z(x,0)=z_0,&&\text{in }\Om,\\
&-\varphi_{t}-\mu\Delta \varphi-(\overline{y}\cdot\nabla)\varphi+(\nabla\overline{y})^{T}\varphi+\nabla\nu=z-(\nabla z)^{T}\overline{q}\\
&\qquad\qquad\qquad\qquad\qquad\qquad\qquad\qquad\qquad\quad+(z\cdot\nabla)\overline{q},&&\text{in }Q_{T},\\
&\Div \ \varphi=0,&&\text{in }Q_{T},\\
&\varphi=0,&&\text{on }\Gamma_{T},\\
&\varphi(x,T)=\varphi_0,&&\text{in }\Om,
\end{aligned}
\right.
\end{align}
where $z_0=y_0-\overline{y}$ and $\varphi_0=q_0-\overline{q}$.

We observe that the right hand side of the equation satisfied by $\varphi$ in \eqref{3.3}, can be written using the definition of $B$ as
\begin{align*}
(\nabla z)^{T}\overline{q}-(z\cdot\nabla)\overline{q}=B'(z)^{*}\overline{q}.
\end{align*}

Since the nonlinear function $B$ is of quadratic nature, we deduce that the derivative of $B'(z)^{*}\overline{q}$ with respect to $z$ is the same function $B'(z)^{*}\overline{q}$. Then, the optimality system \eqref{3.3}, in the references case when $\varphi_0=0$, can be expressed as a linear quadratic optimal control problem, minimizing the functional
\begin{align}\label{oe}
\D L(u)=\frac{1}{2}\int_{Q_{T}}|z|^2 dxdt-\int_{Q_{T}}[(\nabla z)^{T}\overline{q}-(z\cdot\nabla)\overline{q}]dxdt +\frac{1}{2}\int_{0}^{T}\|v(t)\|_{L^2(\Om)}^2dt,
\end{align}
such that $(z,\varphi)$ is the unique solution of
\begin{align*}
\left\{\begin{array}{rllll}
z_{t}-\mu\Delta z+(\overline{y}\cdot\nabla)z+(z\cdot\nabla)\overline{y}+\nabla\eta&=&v&,&\text{in }Q_{T},\\
\Div \ z&=&0&,&\text{in }Q_{T},\\
z&=&0&,&\text{on }\Gamma_{T},\\
z(x,0)&=&z_0&,&\text{in }\Om.
\end{array}
\right.
\end{align*}

For our purposes, we need to give the basic hypothesis such that the optimal control problem for Oseen equation \eqref{oe} satisfies the turnpike property. To ensure this, we will use the result of Porretta and Zuazua \cite{porretta2013long}. In this paper the authors prove the turnpike property for linear problems.

Consider the control problem for Oseen equation
\begin{align}\label{4.6}
\left\{\begin{array}{rllll}
z_{t}+\mathcal{A}z&=&v& ,& \text{in }(0,T),\\
z(0)&=&z_0&,&
\end{array}
\right.
\end{align}
where $\mathcal{A}$ is the Oseen operator defined by \eqref{4.2} and the control $v$ is in $L^2(0,T; H)$. 

It is easy to prove that the Oseen operator satisfies
\begin{align}\label{4.7}
\exists \gamma,\xi>0\; :\quad \langle \mathcal{A}z,z\rangle_{V',V}+\gamma \|z\|_{H}^2\geq\xi \|z\|_{V}^{2}\;,\forall x\in V.
\end{align}

Also, if we assume that the initial data $z_0$ is in $X_{\sigma}$ and $\sigma>0$ satisfies \eqref{4.4}, we obtain by Theorem \ref{teo9} that the semigroup associated to the Oseen equation decays exponentially. 

Then, there exists $C>0$ such that for every solution $z$ of \eqref{4.6} and $z_0\in X_{\sigma}$, we have
\begin{align}\label{4.8}
\|z(T)\|_{H}\leq C\left(\|z_0\|_{H}+\int_{0}^{T}\|v(s)\|_{V'}ds\right).
\end{align}

Besides, from the paper of Fursikov \cite{fursikov2012feedback} we know that there exists a linear bounded operator $L:V\to V$ such that the control $v(t,\cdot)$ can be expressed by
\begin{align*}
v(t,\cdot)=Lz(t,\cdot)
\end{align*}
with the solution of \eqref{4.6} satisfying 
\begin{align}\label{4.9}
\D\|z(t,\cdot)\|_{V}\leq c\|z_0\|_{V}\; e^{-\sigma t},\quad \text{for }t\geq 0.
\end{align}

Assuming that the tracking term $\|\overline{y}-x^{d}\|_{V}$ is sufficiently small, the viscosity function satisfies $\mu>\mathcal{M}(\overline{y})$, and $z_0\in X_{\sigma}$, reasoning as in the proof of Theorem \ref{teo8}, we deduce that the functional $L$ is coercive. This implies, by Theorem $3.10$ in \cite{porretta2013long}, that the optimality system \eqref{3.3} satisfies the turnpike property. Namely, 
\begin{align*}
\|z^{T}(t)\|_{L^2(\Om)}+\|\varphi^{T}(t)\|_{L^2(\Om)}\leq C(e^{-\gamma t}+e^{-\gamma(T-t)})\;,\quad \forall t\in(0,T).
\end{align*}

Then, as in \cite{porretta2013long}, we can define a linear bounded operator in $(L^2(\Om))^2$ as
\begin{align*}
P(T)z_0=\varphi(0)
\end{align*}
such that
\begin{align}\label{3.8}
\|P(t)-\hat{P}\|_{\mathcal{L}((L^2(\Om))^2,(L^2(\Om))^2)} \leq C e^{-2\gamma t},
\end{align}
for some constant $C>0$ and $\gamma>0$. $\hat{P}$ being the corresponding operator for the infinite horizon control problem.

Using the previous turnpike property for Oseen equation, we can state and prove the main theorem of this paper. 

\begin{thm}\label{teo10}
We assume that the tracking term $\|\overline{y}-x^{d}\|_{V}$ is sufficiently small, $\mu>\mathcal{M}(\overline{y})$, and $z_0=y_0-\overline{y}\in X_{\sigma}$. Then, there exists some $\epsilon>0$ such that for every $y_0,q_0$ with
\begin{align*}
\|y_0-\overline{y}\|_{L^2(\Om)}+\|q_0-\overline{q}\|_{L^2(\Om)}\leq \epsilon,
\end{align*}
there exists a solution of the optimality system \eqref{3.1} such that
\begin{align}\label{3.4}
\|y^{T}(t)-\overline{y}\|_{L^2(\Om)}+\|q^{T}(t)-\overline{q}\|_{L^2(\Om}\leq C(e^{-\gamma t}+e^{-\gamma(T-t)}),\quad \forall t<T,
\end{align}
where $\gamma>0$ is the stabilizing rate of the linearized optimality system \eqref{3.3}.
\end{thm}

\begin{proof}
The proof follows the arguments of \cite{porretta2013long, porretta2014remarks}.

The main idea of the proof is to consider a perturbed problem of \eqref{3.3} and then to implement a fixed point argument, which gives the solutions of the optimality system \eqref{3.1}.

Let $X$ be the set
\begin{align*}
X=\{(z,\varphi)\; :\; \|z\|_{V}+\|\varphi\|_{V}\leq M(e^{-\gamma t}+e^{-\gamma(T-t)}), \; \forall t\in[0,T]\},
\end{align*}
for some $M\leq 1$. For $(\hz,\hp)\in X$, we consider
\begin{align*}
R_1(\hz)=-(\hz\cdot\nabla)\hz
\end{align*}
and
\begin{align*}
R_2(\hz,\hp)=(\hz\cdot\nabla)\hp-(\hp\cdot\nabla)\hz.
\end{align*}

Note that the terms $R_1$ and $R_2$ can be expressed in an abstract way, namely
\begin{align*}
R_1(\hz)=-B(\hz)\quad,\quad R_2(\hz,\hp)=-B'(\hz)^{*}\hp.
\end{align*}

Then, using the properties for the nonlinear form $B$, we obtain that
\begin{align}\label{3.5}
\begin{array}{c}
\|R_1(\hz)(t)\|_{L^2(\Om)}\leq\|R_1(\hz)(t)\|_{V}\leq c_0M^2(e^{-2\gamma t}+e^{-2\gamma(T-t)}),\\[0.2em]
\|R_2(\hz,\hp)(t)\|_{L^2(\Om)}\leq\|R_2(\hz,\hp)(t)\|_{V}\leq c_1M^2(e^{-2\gamma t}+e^{-2\gamma(T-t)}),
\end{array}
\end{align}
where $c_1$ depend on $\|\overline{y}\|_{V}$.

Besides, we define the operator
\begin{align}\label{3.6}
R(\hz,\hp)=(z,\varphi),
\end{align}
where $(z,\varphi)$ solve the problem
\begin{align}\label{3.7}
\left\{\begin{aligned}
&z_{t}-\mu\Delta z+(\overline{y}\cdot\nabla)z+(z\cdot\nabla)\overline{y}+\nabla\eta=-\varphi+R_1(\hz),&&\text{in }Q_{T},\\
&\Div \ z=0,&&\text{in }Q_{T},\\
&z=0,&&\text{on }\Gamma_{T},\\
&z(x,0)=z_0,&&\text{in }\Om,\\
&-\varphi_{t}-\mu\Delta \varphi-(\overline{y}\cdot\nabla)\varphi+(\nabla\overline{y})^{T}\varphi+\nabla\nu=z-(\nabla z)^{T}\overline{q}\\
&\qquad\qquad\qquad\qquad\qquad\qquad\qquad\qquad\qquad\quad+(z\cdot\nabla)\overline{q}+R_2(\hz,\hp),&&\text{in }Q_{T},\\
&\Div \ \varphi=0,&&\text{in }Q_{T},\\
&\varphi=0,&&\text{on }\Gamma_{T},\\
&\varphi(x,T)=\varphi_0,&&\text{in }\Om,
\end{aligned}
\right.
\end{align}

Then, we need to prove that the operator $R$ has a fixed point which is a solution of \eqref{3.1} and satisfies the estimate \eqref{3.4}.

Define $h$ as a solution of the equation
\begin{align}\label{3.9}
\left\{\begin{aligned}
&-h_{t}-\mu\Delta h-(\overline{y}\cdot\nabla)h+(\nabla\overline{y})^{T}h\\
&\qquad+\nabla\nu+P(T-t)h=P(T-t)R_1(\hz)+R_2(\hz,\hp),&&\text{in }Q_{T},\\
&\Div \ h=0,&&\text{in }Q_{T},\\
&h=0,&&\text{on }\Gamma_{T},\\
&h(x,T)=\varphi_0,&&\text{in }\Om.
\end{aligned}
\right.
\end{align}

Then, it is easy to prove that $h$ satisfies 
\begin{align}\label{3.10}
h=\varphi-P(T-t)z
\end{align}
in a weak sense, namely for all test function $\phi$
\begin{align*}
\D\int_{\Om}h(t)\phi=\int_{\Om}\varphi(t)\phi dx-\int_{\Om}z(t)[P(T-t)\phi]dx.
\end{align*}

We observe that $h$ can be estimated as
\begin{align*}
\D h(t)=&e^{-\mathcal{M}(T-t)}\varphi_0+\int_{t}^{T}e^{\mathcal{M}(t-s)})[P(T-s)R_1(\hz)(s)+R_2(\hz,\hp)(s)]ds\\
&-\int_{t}^{T}e^{\mathcal{M}(t-s)}[\hat{P}-P(T-s)]h(s)ds,
\end{align*}
where $\mathcal{M}v=-\mu\Delta v-(\overline{y}\cdot\nabla)v+(\nabla\overline{y})^{T}v+\hat{P}$. We observe that $\mathcal{M}$ is exponentially stable with rate $\gamma$. Using the estimates \eqref{3.8} and \eqref{3.5}, we obtain  
\begin{align*}
\D \|h(t)\|_{L^2(\Om)}\leq &\;e^{-\gamma(T-t)}\|\varphi_0\|_{L^2(\Om)}+cM^2\int_{t}^{T}e^{\gamma(t-s)}(e^{-2\gamma s}+e^{-2\gamma(T-s)})ds\\
&+\int_{t}^{T}e^{\gamma(t-s)}e^{-2\gamma(T-s)}\|h(s)\|_{L^2(\Om)}ds\\
\leq &\; e^{-\gamma(T-t)}\|\varphi_0\|_{L^2(\Om)}+cM^2[e^{-2\gamma t}+e^{-\gamma(T-t)}]\\
&+\int_{t}^{T}e^{-2\gamma T+\gamma t+\gamma s}\|h(s)\|_{L^2(\Om)}ds.
\end{align*}

By the Gronwall inequality
\begin{align*}
\D \|h(t)\|_{L^2(\Om)}\leq &\;e^{-\gamma(T-t)}\|\varphi_0\|_{L^2(\Om)}&\\
& +cM^2[e^{-2\gamma t}+e^{-\gamma(T-t)}]\D \text{exp}\left(\D\int_{t}^{T}e^{-2\gamma T+\gamma t+\gamma s}ds\right).
\end{align*}

The last integral can be estimated easily by $\frac{1}{\gamma}$. Therefore
\begin{align}\label{3.11}
\|h(t)\|_{L^2(\Om)}\leq &\;e^{-\gamma(T-t)}[\|\varphi_0\|_{L^2(\Om)}+cM^2] +cM^2 e^{-2\gamma t}.
\end{align}

From the estimate for $h$ we can find a similar estimate for $z$ and $\varphi$. Indeed, observe that $z$ satisfies the following equation
\begin{align*}
z_t-\mu\Delta z+(\overline{y}\cdot\nabla)z+(z\cdot\nabla)\overline{y}+\hat{P}+\nabla\eta=(\hat{P}-P(T-t))z-h+R_1(\hz).
\end{align*}

Therefore, we obtain that
\begin{align*}
\D z(t)= e^{-\mathcal{N}t}z_0+\int_{0}^{t}e^{-\mathcal{N}(t-s)}&[\hat{P}-P(T-s)]z(s)ds&\\
&+\int_{0}^{t}e^{-\mathcal{N}(t-s)}(R_1(\hz)(s)-h(s))ds,
\end{align*}
where $\mathcal{N}v=-\mu\Delta v+(\overline{y}\cdot\nabla)v+(v\cdot\nabla)\overline{y}+\hat{P}$. We note that $\mathcal{N}$ satisfies the exponentially decay with rate $\gamma$. Again, using the estimate \eqref{3.5}, \eqref{3.8}, and \eqref{3.11} we get
\begin{align*}
\D\|z(t)\|_{L^2(\Om)}\leq&\; e^{-\gamma t}\|z_0\|_{L^2(\Om)}+\int_{0}^{t}e^{-\gamma(t-s)}e^{-2\gamma(T-s)}\|z(s)\|_{L^2(\Om)}ds\\
&+cM^2\int_{0}^{t}e^{-\gamma(t-s)}(e^{-2\gamma(T-s)}+e^{-2\gamma s})ds\\
&+\int_{0}^{t}e^{-\gamma(t-s)}[e^{-\gamma(T-s)}(\|\varphi_0\|_{L^2(\Om)}+cM^2)+cM^2e^{-2\gamma s}]ds\\
\leq&\;  e^{-\gamma t}\|z_0\|_{L^2(\Om)}+cM^2[e^{-2\gamma(T-t)}+e^{-\gamma t}]\\
&+[\|\varphi_0\|_{L^2(\Om)}+cM^2]e^{-\gamma(T-t)}\\
&+\int_{0}^{t}e^{-2\gamma T-\gamma t+3\gamma s}\|z(s)\|_{L^2(\Om)}ds.
\end{align*}

Applying again the Gronwall inequality, we obtain 
\begin{align}\label{3.12}
\|z(t)\|_{L^2(\Om)}\leq [\|z_0\|_{L^2(\Om)}+\|\varphi_0\|_{L^2(\Om)}+cM^2](e^{-\gamma t}+e^{-\gamma(T-t)}).
\end{align}

Using now that $\varphi=h+P(T-t)z$, we get an estimate for $\varphi$
\begin{align}\label{3.13}
\|\varphi(t)\|_{L^2(\Om)}\leq [\|z_0\|_{L^2(\Om)}+\|\varphi_0\|_{L^2(\Om)}+cM^2](e^{-\gamma t}+e^{-\gamma(T-t)}).
\end{align}

Now, we go back on the first equation of \eqref{3.7}. Observe that
\begin{align}\label{3.14}
\|-\varphi(t)+R_1(\hz)(t)\|_{L^2}\leq  [\|z_0\|_{L^2}+\|\varphi_0\|_{L^2}+cM^2](e^{-\gamma t}+e^{-\gamma(T-t)}).
\end{align}

Then, by the regularity of the solution of the linearized problem, see Lemma \ref{lem1}, we have that
\begin{align}
\|z(t)\|_{H^2(\Om)}\leq [\|z_0\|_{L^2(\Om)}+\|\varphi_0\|_{L^2(\Om)}+cM^2](e^{-\gamma t}+e^{-\gamma(T-t)}).
\end{align}

And, we can conclude that
\begin{align*}
\|z(t)\|_{V}\leq [\|z_0\|_{L^2(\Om)}+\|\varphi_0\|_{L^2(\Om)}+cM^2](e^{-\gamma t}+e^{-\gamma(T-t)}).
\end{align*}

Analogously, we obtain the same estimate for $\varphi$, namely
\begin{align*}
\|\varphi(t)\|_{V}\leq [\|z_0\|_{L^2(\Om)}+\|\varphi_0\|_{L^2(\Om)}+cM^2](e^{-\gamma t}+e^{-\gamma(T-t)}).
\end{align*}

Finally, we choose $M\leq 1$ such that $cM^2\leq \frac{M}{2}$. Then, if we assume that the initial and final state are close enough to the stationary primal and dual state, respectively, we obtain that 
\begin{align*}
c[\|z_0\|_{L^2(\Om)}+\|\varphi_0\|_{L^2(\Om)}+M^2]\leq M.
\end{align*}

So, we deduce that the space $X$ becomes an invariant convex subset of $L^2(0,T;(L^2(\Om))^2)$. Besides, we observe that operator $R$ is continuous and compact, then we conclude the existence of a fixed point $(z,\varphi)$ of $R$. It is easy to see that $(z,\varphi)$ is a solution of the optimality system \eqref{3.1}. Then the proof is complete.
\end{proof}

\begin{rem}
Since we develop a local analysis around the optimal solution for the stationary problem, the turnpike property for Oseen equation is fundamental in our work. In this point is fundamental the smallness assumption on the tracking term. If we remove the last condition, we need to suppose that the optimality system \eqref{3.3} satisfy the turnpike property to ensure our result.

An interesting problem is to prove the necessary and sufficient conditions to obtain the turnpike property for the linearized optimality systems.
\end{rem}

\setcounter{equation}{0}
\section{Turnpike property for the two-dimensional Navier-Stokes problem with time-independent control}

In this section we prove a turnpike property for the two-dimensional Navier-Stokes problem in the particular case when the controls are independent on time. 

The proof is different from that given in the previous section when the control function depends on time. In this case, we obtain the result using the classical $\Gamma$-convergence, and a standard stability property of the Navier-Stokes equation, see Theorem \ref{teo3.4}, under suitable conditions of smallness of the data.

That technique is a general principle proved by Porretta and Zuazua \cite{porretta2014remarks} for the semilinear heat equation. Of course, can also be employed for a larger class of semilinear problems enjoying standard exponentially stability. 

Let $\Om\subset \R^2$ be a bounded and simply connected domain, with boundary $\pa\Om$ of class $C^2$. We consider the Navier-Stokes control problem
\begin{align}\label{72}
\left\{\begin{array}{rllll}
y_{t}-\mu\Delta y+(y\cdot\nabla)y+\nabla p&=&u(x)& ,& \text{in }Q_{T},\\
\Div \ y&=&0&,&\text{in }Q_{T},\\
y&=&0&,&\text{on }\Gamma_{T},\\
y(x,0)&=&y_0(x)&,& x\in\Om,
\end{array}
\right.
\end{align}
with controls $u=u(x)$ independent of time. 

We consider the optimal control problem
\begin{align}\label{73}\left\{\begin{array}{c}
\D\min J^{T}(u)=\frac{1}{2}\int_{0}^{T}\|y(t)-z\|_{L^2(\Om)}^2 dt+\frac{T}{2}\|u\|_{L^2(\Om)}^2,\\[1em]
\text{s. a. } \qquad y  \text{ solution of \eqref{72} and }u\in C,
\end{array}
\right.
\end{align}
where $C$ is a closed convex subset of $(L^2(\Om))^2$ and $z\in (L^2(\Om))^2$ denotes the desired state.

In addition, we consider the analogous stationary optimal control problem
\begin{align}\label{74}
\left\{\begin{array}{rllll}
-\mu\Delta y+(y\cdot\nabla)y+\nabla p&=&u(x)& ,& \text{in }\Om,\\
\Div \ y&=&0&,&\text{in }\Om,\\
y&=&0&,&\text{on }\pa\Om,\\
\end{array}
\right.
\end{align}
together with the corresponding functional 
\begin{align}\label{75}\left\{\begin{array}{c}
\D\min J(u)=\frac{1}{2}\left(\|y-z\|_{L^2(\Om)}^2+\|u\|_{L^2(\Om)}^2\right),\\[1em]
\text{s. a. } \qquad y  \text{ solution of \eqref{74} and }u\in C.
\end{array}
\right.
\end{align}

In both cases, we consider that $C$ has the following form
\begin{align*}
C\equiv U_{ad}:=\{u\in (L^2(\Om))^2\;:\; \|u\|\leq c(\Om)\mu^2,\;\forall x\in\Om\}.
\end{align*}

\begin{rem}
In view of the Theorems $\ref{teo2}$ and $\ref{teo6}$, we note that in both cases the optima are achieved.  The only difference is that in this case we consider constrains on the controls. However, since $U_{ad}$ is a convex closed subset of $(L^2(\Om))^2$, we assert the result using the classical results of convex analysis.
\end{rem}

For a given source term $u$ which does not depend on time, we consider a steady solution $(y_{\infty},p_{\infty})\in ((H^2(\Om))^2\cap V)\times (H^1(\Om)\cap L_0^2(\Om))$ to the stationary Navier-Stokes problem. Then, the solution $(y,p)$ to \eqref{72} converge to $(y_{\infty},p_{\infty})$ as $t\to\infty$, under suitable assumptions.

\begin{thm}\label{teo3.4}
There exists $C>0$ and $\alpha>0$ depending only on $\Om$ such that, under the condition 
\begin{align*}
\D\|\nabla y_{\infty}\|_{L^2(\Om)}\leq C\mu,
\end{align*}
there exists a unique weak solution $(y,p)$ of \eqref{72} which satisfies
\begin{align}\label{78.1}
\|y(t)-y_{\infty}\|_{L^2(\Om)}\leq \|y_0-y_{\infty}\|_{L^2(\Om)}e^{-\alpha t},\quad \forall t\geq 0.
\end{align}

\end{thm}
\begin{proof}
Let $(y,p)$ and $(y_{\infty},p_{\infty})$ be the solution of the evolutionary and stationary Navier-Stokes problem, respectively. Let $y=y_{\infty}+w$ and $p=p_{\infty}+q$, where $(w,q)$ solves the problem
\begin{align}\label{78}
\left\{\begin{aligned}
&w_{t}-\mu\Delta w+(w\cdot\nabla)w+(y_{\infty}\cdot\nabla)w+(w\cdot\nabla)y_{\infty}+\nabla q=0 ,&& \text{in }Q_{T},\\
&\Div \ w=0,&&\text{in }Q_{T},\\
&w=0,&&\text{on }\Gamma_{T},\\
&w(x,0)=y_0(x)-y_{\infty},&& x\in\Om.
\end{aligned}
\right.
\end{align}

Multiplying the equation \eqref{78} by $w$ and using the definition of $b$, we obtain
\begin{align*} 
\D\int_{\Om}w_{t}\;w dx-\mu\int_{\Om}\nabla w \;w+b(w,w,w)+b(y_{\infty},w,w)+b(w,y_{\infty},w)=0
\end{align*}
and by Lemma \ref{lema3.1}, we deduce that
\begin{align*}
\D\frac{1}{2}\frac{d}{dt}\|w(t)\|_{L^2(\Om)}^2+\mu\|\nabla w(t)\|_{L^2(\Om)}^2\leq C \|\nabla y_{\infty}\|_{L^2(\Om)}\|w(t)\|_{L^2(\Om)}\|\nabla w(t)\|_{L^2(\Om)}.
\end{align*}

We remind the following Young inequality
\begin{align*}
x_1\cdot\ldots\cdot x_{n}\leq e_1 x_1^{p_1}+\ldots+e_{n-1}x_{n-1}^{p_{n-1}}+C(e_1,\ldots,e_{n-1})x_{n}^{p_{n}},
\end{align*}
where $p_1^{-1}+\ldots+p_{n}^{-1}=1$ and $e_1,\ldots,e_{n-1}, x_1,\ldots,x_{n}$. are positive real numbers.

Then, using the Young inequality for $x_1=\|\nabla y_{\infty}\|_{L^2(\Om)}\|w(t)\|_{L^2(\Om)}$,  \\$x_2=\|\nabla w(t)\|_{L^2(\Om)}$, $e_1=\frac{1}{2\mu}$, $e_2=\frac{\mu}{2}$ and $p_1=p_2=2$, we obtain
\begin{align*}
\D\frac{d}{dt}\|w(t)\|_{L^2(\Om)}^2+\mu\|\nabla w(t)\|_{L^2(\Om)}^2\leq C\frac{1}{\mu} \|\nabla y_{\infty}\|_{L^2(\Om)}^2\|w(t)\|_{L^2(\Om)}^2.
\end{align*}

From the Poincar\'e inequality for the Stokes operator, we have that
\begin{align*}
\D\frac{d}{dt}\|w(t)\|_{L^2(\Om)}^2+\left(C_1\mu-\frac{C}{\mu}\|\nabla y_{\infty}\|_{L^2(\Om)}^2\right)\|w(t)\|_{L^2(\Om)}^2\leq 0.
\end{align*}

Provided that $\D\frac{C}{\mu^2}\|\nabla y_{\infty}\|_{L^2(\Om)}^2\leq C_1 $, we have
\begin{align*}
\D\frac{d}{dt}\|w(t)\|_{L^2(\Om)}^2+2\alpha\|w(t)\|_{L^2(\Om)}^2\leq0,
\end{align*}
which finally gives 
\begin{align*}
\|y(t)-y_{\infty}\|_{L^2(\Om)}^2\leq \|y_0-y_{\infty}\|_{L^2(\Om)}^2e^{-2\alpha t},\quad \forall t\geq 0.
\end{align*}
\end{proof}

Now, we can prove the following turnpike result for controls independent of time.
\begin{thm}\label{teo15}
Assume that the hypotheses of the Theorems $\ref{teo1}$ and $\ref{teo5}$ holds. Let $(y^{T_{n}},u^{T_{n}})$ be an optimal solution of \eqref{73} for $T=T_{n}$. Then any accumulation point $(y_{\infty},u_{\infty})$, as $n\to\infty$, is an optimal solution of \eqref{75}.
\end{thm}

\begin{proof}
The proof is based on arguments similar to those used in \cite{porretta2014remarks}. The main idea of the proof is to use the $\Gamma$-convergence, since we consider the control function independent of time. 

Then, let $(T_{n})_{n\in\N}$ be an increasing sequence of times converging to infinity. For each $n\in\N$, by Theorem \ref{teo2} the optimal control problem \eqref{73} has at least a minimizer $(y^{T_{n}},u^{T_{n}})\in U_{ad}$. In particular, $(u^{T_{n}})$ is uniformly bounded in $(L^2(\Om))^2$, so we can extract a subsequence, still labeled by n, such that 
\begin{align*}
u^{T_{n}}\rightharpoonup u_{\infty}, \text{ weakly in }L^2(\Om), \text{ as }n\to\infty.
\end{align*}

We claim that 
\begin{multline}\label{84}
\D\lim_{n\to\infty}\frac{1}{T_{n}}\left(\frac{1}{2}\int_{0}^{T_{n}}\|y^{T_{n}}(t)-z\|_{L^2(\Om)}^2 dt+\frac{T_{n}}{2}\|u^{T_{n}}\|_{L^2(\Om)}^2\right)=\\
\frac{1}{2}\left(\|y_{\infty}-z\|_{L^2(\Om)}^2+\|u_{\infty}\|_{L^2(\Om)}^2\right),
\end{multline}
where $y_{\infty}$ solves
\begin{align}\label{85}
\left\{\begin{array}{rllll}
-\mu\Delta y_{\infty}+(y_{\infty}\cdot\nabla)y_{\infty}+\nabla p_{\infty}&=&u_{\infty}(x)& ,& \text{in }\Om,\\
\Div \ y_{\infty}&=&0&,&\text{in }\Om,\\
y_{\infty}&=&0&,&\text{on }\pa\Om.\\
\end{array}
\right.
\end{align}

Observe that the previous limit is equivalent to saying that if we consider $I^{T_{n}}$ and $I$ the values of the minimizers for the time dependent problem in $[0,T_{n}]$ and the steady state, respectively, then 
\begin{align*}
\D\lim_{n\to\infty}\frac{I^{T_{n}}}{T_{n}}=I.
\end{align*}

Indeed, if $\|\nabla \overline{y}^{T_{n}}\|_{L^2(\Om)}\leq C\mu$ by Theorem \ref{teo3.4} we have that
\begin{align}\label{86}
\|y^{T_{n}}-\overline{y}^{T_{n}}\|_{L^2(\Om)}\leq \|y_{\infty}-\overline{y}^{T_{n}}\|_{L^2(\Om)}e^{-\alpha t},\;\forall t>0,
\end{align}
where $\nabla \overline{y}^{T_{n}}$ satisfies
\begin{align}\label{87}
\left\{\begin{array}{rllll}
-\mu\Delta \overline{y}^{T_{n}}+(\overline{y}^{T_{n}}\cdot\nabla)\overline{y}^{T_{n}}+\nabla \overline{p}^{T_{n}}&=&u^{T_{n}}(x)& ,& \text{in }\Om,\\
\Div \ \overline{y}^{T_{n}}&=&0&,&\text{in }\Om,\\
\overline{y}^{T_{n}}&=&0&,&\text{on }\pa\Om.\\
\end{array}
\right.
\end{align}

We observe that from the regularity of the stationary Navier-Stokes problem, we have that
$$
\|\overline{y}^{T_{n}}\|_{H^2(\Om)}+\|\overline{p}^{T_{n}}\|_{H^1(\Om)}\leq C(1+\|u^{T_{n}}\|_{L^2(\Om)}^3),
$$
and in particular
$$
\|\nabla \overline{y}^{T_{n}}\|_{L^2(\Om)}\leq C\|u^{T_{n}}\|_{L^2(\Om)}.
$$

Then, $\|\nabla \overline{y}^{T_{n}}\|_{L^2(\Om)}\leq C\mu$. By Theorem \ref{teo3.4}, we obtain the estimate \eqref{86}.

Now, we decompose
$$
\D\frac{J^{T_{n}}(u^{T_{n}})}{T_{n}}-J(u_{\infty})=J_1^{n}+J_2^{n},
$$
where
\begin{align}\label{88}
J_1^{n}=\D\frac{J^{T_{n}}(u^{T_{n}})}{T_{n}}-J(u^{T_{n}}),
\end{align}
and
\begin{align}\label{89}
J_2^{n}=J(u^{T_{n}})-J(u_{\infty}).
\end{align}

We study the convergence of $J_{1}^{n}$ and $J_{2}^{n}$ as $n\to\infty$. First, we analyze $J_{1}^{n}$:
\begin{align*}
J_{1}^{n}=&\frac{1}{T_{n}}\left(\frac{1}{2}\int_{0}^{T_{n}}\|y^{T_{n}}(t)-z\|_{L^2(\Om)}^2 dt+\frac{T_{n}}{2}\|u^{T_{n}}\|_{L^2(\Om)}^2\right)\\
&-\frac{1}{2}\|\overline{y}^{T_{n}}-z\|_{L^2(\Om)}^2-\frac{1}{2}\|u^{T_{n}}\|_{L^2(\Om)}^2\\
=&\frac{1}{2T_{n}}\int_{0}^{T_{n}}\|y^{T_{n}}(t)-z\|_{L^2(\Om)}^2 dt-\frac{1}{2}\|\overline{y}^{T_{n}}-z\|_{L^2(\Om)}^2.
\end{align*}

Since $u^{T_{n}}$ is uniformly bounded in $(L^2(\Om))^2$, from the regularity of the Navier-Stokes problem, we obtain that $\overline{y}^{T_{n}}$ is uniformly bounded in $(H^2(\Om))^2\cap V$, in particular, in $(L^2(\Om))^2$. Then, using again the exponential stability property \eqref{86}, we deduce that
$$
I_{1}^{n}\to0\quad \text{as }n\to\infty.
$$

For $I_2^{n}$ we have
\begin{align*}
I_2^{n}&=J(u^{T_{n}})-J(u_{\infty})\\
&=\frac{1}{2}\|\overline{y}^{T_{n}}-z\|_{L^2(\Om)}^2+\frac{1}{2}\|u^{T_{n}}\|_{L^2(\Om)}^2-\frac{1}{2}\|y_{\infty}-z\|_{L^2(\Om)}^2-\frac{1}{2}\|u_{\infty}\|_{L^2(\Om)}^2.
\end{align*}

Since $\overline{y}^{T_{n}}$ is bounded in $(H^2(\Om))^2\cap V$, there exists some $y^{*}\in (H^2(\Om))^2\cap V$ and a subsequence of $\overline{y}^{T_{n}}$ such that
$$
\overline{y}^{T_{n}}\rightharpoonup y^{*}\text{ weakly in }H^2(\Om)\cap V.
$$

We know that the injection of $V$ into $(L^2(\Om))^2$ is compact, so we have also
$$
\overline{y}^{T_{n}}\to y^{*}\text{ in the norm of }L^2(\Om). 
$$

Besides, the trilinear function $b$ is continuous and by the Lemma $1.5$, chapter II in \cite{temam2001navier}, we obtain that $b(\overline{y}^{T_{n}},\overline{y}^{T_{n}},v)\to b(y^{*},y^{*},v)$, for all $v\in V$. Finally, since $u^{T_{n}}$ converge to $u_{\infty}\in U_{ad}$, by the uniqueness of the Navier-Stokes problem, we obtain that $y^{*}=y_{\infty}$. Therefore, we conclude that 
$$
I_{2}^{n}\to 0,\quad \text{as }n\to\infty.
$$

This completes the proof of the claim.

Now, we need to prove that $u_{\infty}$ is an optimal solution of \eqref{75}. Indeed, by the weak convergence of $u^{T_{n}}$ we obtain that
$$
\D\|u_{\infty}\|_{L^2(\Om)}\leq\liminf_{n\to\infty}\|u^{T_{n}}\|_{L^2(\Om)}.
$$

Also, we have that $\overline{y}^{T_{n}}$ weakly converges to $y_{\infty}$, as $n\to\infty$. Then,
\begin{align*}
\D\left\|\frac{\int_{0}^{T_{n}}  y^{T_{n}}(t)dt}{T_{n}}-y_{\infty}\right\|_{L^2(\Om)}\leq \left\|\frac{\int_{0}^{T_{n}}y^{T_{n}}(t)dt}{T_{n}}-\overline{y}^{T_{n}}\right\|_{L^2(\Om)}+\|\overline{y}^{T_{n}}-y_{\infty}\|_{L^2(\Om)},
\end{align*}
and we obtain that
$$
\frac{\int_{0}^{T_{n}}y^{T_{n}}(t)dt}{T_{n}}\to y_{\infty} \text{ in }L^2(\Om),\text{ as }n\to\infty.
$$

Then, necessarily we have
$$
J(u_{\infty})\leq \liminf_{n\to\infty}\frac{J^{T_{n}}(u^{T_{n}})}{T_{n}}.
$$

Therefore, using the claim \eqref{84}, we obtain
$$
J(u_{\infty})\leq I
$$
and finally, $u_{\infty}$ is a minimizer for the steady state problem, with $y_{\infty}$ the associated state.
\end{proof}

\begin{rem}
We observe that the proof of the Theorem $\ref{teo15}$ uses the exponential stabilization result (Theorem $\ref{teo3.4}$) in many ocasions. We know that in the three-dimensional case this property is also true for strong solutions, but under more smallness condition of the stationary solutions. This implies that the three-dimensional case is more complex that the two-dimensional problem. 
\end{rem}

\begin{rem}
Note that we considered the $L^2$-norm in the tracking term on the functional to minimize. However, in the three-dimensional case, this choice is not correct because there is no way to assure the optimal state to be a strong solution of the evolutionary Navier-Stokes problem. The good choice would be, for instance \cite{casas1998optimal},
\begin{align*}
J(u)=\D\frac{1}{8}\int_{0}^{T}\left(\int_{\Om}|y-x^{d}|^{4}dx\right)^2 dt +\frac{T}{2}\|u\|_{L^2(\Om)},
\end{align*}
with $x^{d}\in L^8([0,T]; (L^4(\Om))^3)$. 
\end{rem}

\section*{Acknowledgment}

I wish to express my gratitude to Prof. Enrique Zuazua, who proposed me this problem and for several interesting suggestions and comments.


\begin{thebibliography}{99}

\bibitem{abergel1993some}{\sc F. Abegel and E. Casas}. {\it Some optimal control problems of multistate equations appearing in fluid mechanics}. RAIRO-Mod{\'e}lisation math{\'e}matique et analyse num{\'e}rique, 27(2):223--246, 1993.


\bibitem{abergel1990some}{\sc F. Abergel and R. Temam}. {\it On some control problems in fluid mechanics}. Theoretical and Computational Fluid Dynamics, 1(6):303--325, 1990. 

\bibitem{allaire2010long}{\sc G. Allaire, A. M{\"u}nch, and F. Periago}. {\it Long time behavior of a two-phase optimal design for the heat equation}. SIAM Journal on Control and Optimization, 48(8):5333--5356, 2010.


\bibitem{barbu2003feedback}{\sc V. Barbu}. {\it Feedback stabilization of Navier--Stokes equations}. ESAIM: Control, Optimisation and Calculus of Variations, 9:197--205, 2003.

\bibitem{barbu2011stabilization}{\sc V. Barbu}. {\it Stabilization of Navier--Stokes Flows}. Springer, 2011.

\bibitem{boyer2012mathematical}{\sc F. Boyer and P. Fabrie}. {\it Mathematical tools for the study of the incompressible Navier--Stokes equations and related models}, volume 183. Springer Science \& Business Media, 2012.

\bibitem{casas1998optimal}{\sc E. Casas}. {\it An optimal control problem governed by the evolution Navier--Stokes equations}. Optimal control of viscous flow, 59:79--95, 1998.

\bibitem{casas2007error}{\sc E. Casas, M. Mateos, and J.-P. Raymond}. {\it Error estimates for the numerical approximation of a distributed control problem for the steady-state Navier--Stokes equations}. SIAM Journal on Control and Optimization, 46(3):952--982, 2007. 

\bibitem{de2004primal}{\sc J. De los Reyes}. {\it A Primal-Dual Active Set Method for Bilaterally Control Constrained Optimal Control of the Navier--Stokes Equations}. Numerical functional analysis and optimization, 25(7):657--684, 2004.

\bibitem{fursikov2001stabilizability}{\sc A. Fursikov}. {\it Stabilizability of Two-Dimensional Navier--Stokes Equations with Help of a Boundary Feedback Control}. Journal of Mathematical Fluid Mechanics, 3(3):259--301, 2011.

\bibitem{fursikov2012feedback}{\sc A. Fursikov}. {\it Feedback stabilization for Navier--Stokes equations: theory and calculations}. Mathematical Aspects of Fluid Mechanics, 402:130--172, 2012.

\bibitem{galdi2013introduction}{\sc G. P. Galdi}. {\it An Introduction to the Mathematical Theory of the Navier--Stokes Equations: Volume I: Linearized Steady Problems}, volume 38. Springer Science \& Business Media, 2013.

\bibitem{hinze2000optimal}{\sc M. Hinze}. {\it Optimal and instantaneous control of the instationary Navier--Stokes equations}. PhD thesis, Habilitation thesis, Technische Universit\"at Berlin, 2000.

\bibitem{huan1994optimum}{\sc J. Huan and V. Modi}. {\it Optimum design of minimum drag bodies in incompressible laminar flow using a control theory approach}. Inverse Problems in Engineering, 1(1):1--25, 1994.

\bibitem{jameson1998optimum}{\sc A. Jameson, L. Martinelli, and N. Pierce}. {\it Optimum Aerodynamic Design Using the Navier--Stokes Equations}. Theoretical and Computational Fluid Dynamics, 1(10), 1998.

\bibitem{jameson2010optimization}{\sc A. Jameson and K. Ou}. {\it Optimization methods in computational fluid dynamics}. Encyclopedia of Aerospace Engineering, 2010.

\bibitem{porretta2013long}{\sc A. Porreta and E. Zuazua}. {\it Long time versus steady state optimal control}, SIAM J. Control Optim., 51(6):4242--4273, 2013.

\bibitem{porretta2014remarks}{\sc A. Porretta and E. Zuazua}. {\it Remarks on long time versus steady state optimal control}. Preprint, 2014.

\bibitem{raymond2006feedback}{\sc J.-P. Raymond}. {\it Feedback boundary stabilization of the two-dimensional Navier--Stokes equations}. SIAM Journal on Control and Optimization, 45(3):790--828, 2006.

\bibitem{temam2001navier}{\sc R. Temam}. {\it Navier-Stokes equations: theory and numerical analysis}, volume 343. American Mathematical Soc., 2001.

\bibitem{trelat2015turnpike}{\sc E. Tr\'elat and E. Zuazua}. {\it The turnpike property in finite-dimensional nonlinear optimal control}. Journal of Differential Equations 258(1):81--114, 2015.

\bibitem{wachsmuth2006optimal}{\sc D. Wachsmuth}. {\it Optimal control of the unsteady Navier-Stokes equations}. PhD thesis, Technische Universit{\"a}t Berlin, 2006.

\bibitem{zowe1979regularity}{\sc J. Zowe and S. kurcyusz}. {\it Regularity and stability for the mathematical programming problem in Banach spaces}. Applied mathematics and Optimization, 5(1):49--62, 1979.





\end{thebibliography}

\end{document}